\documentclass{amsart}
\usepackage{amssymb,amsmath, amsthm,latexsym}
\usepackage{graphics}
\usepackage{todonotes}
\usepackage{amscd}
\usepackage{graphics}
\usepackage{here}
\newcommand{\cal}[1]{\mathcal{#1}}
\theoremstyle{plain}
\newtheorem*{theo}{Theorem}
\newtheorem*{cor}{Corollary}
\newtheorem{lemma}{Lemma}[section]
\newtheorem{theorem}[lemma]{Theorem}
\newtheorem{proposition}[lemma]{Proposition}
\newtheorem{corollary}[lemma]{Corollary}
 
\theoremstyle{definition}

\newtheorem{definition}[lemma]{Definition}
\newtheorem{example}[lemma]{Example}
\newtheorem{remark}[lemma]{Remark}
\parskip=\bigskipamount

\let\phialt=\phi
\let\phi=\varphi
\let\varphi=\phialt

\begin{document}
\title{$L^p$-cohomology for groups of isometries
of Hadamard spaces}
\author{Ursula Hamenst\"adt}
\date{October 24, 2024}

\begin{abstract}
 We show that a discrete group $\Gamma$ which admits a non-elementary isometric
 action on a Hadamard manifold of bounded negative curvature
 admits an isometric action on an $L^p$-space $V$ for some $p>1$ with 
 $H^1(\Gamma,V)\not=0$.
 \end{abstract}

\maketitle

\section{Introduction}

A countable group $\Gamma$ has property (T) if every affine isometric action 
of $\Gamma$ on
an $L^2$-space has a fixed point. Among the most prominent examples
of such groups are lattices in higher rank simple Lie groups.

For such higher rank lattices, much more is true. Namely,
any affine uniformly Lipschitz action on a Hilbert space has a fixed
point \cite{O22, dLdlS23}. Moreover, any isometric action on a uniformly
convex Banach space has a fixed point \cite{BFGM07}, \cite{dLdlS23}.

On the other hand, it is known that any hyperbolic group $\Gamma$ admits
a proper isometric action on some $L^p$-space \cite{Yu05, Ni13}, in spite of
the fact that many such groups have property (T). In particular, 
cocompact lattices in the rank one simple Lie groups
$Sp(2m,1),F_4^ {-20}$ 
admit proper affine isometric action on an $L^p$-space where $p>2$ 
can explicitly be estimated.

Property (T) and its strengthenings can be viewed as a vanishing result
for degree one group cohomology with coefficients in a representation of
$\Gamma$. The goal of this article is to point out that
the construction of representations in an $L^p$-space
with nontrivial first cohomology can be carried out for arbitrary
countable groups which admit
non-elementary isometric actions on Hadamard manifolds 
of bounded negative curvature. Here 
an isometric action 
of a group on a space $Y$ which is hyperbolic in the sense of 
Gromov is \emph{elementary} if either it has a bounded orbit
or if its action on the Gromov boundary of $Y$ has a global fixed point.

\begin{theo}\label{main}
Let $\Gamma$ be a discrete group which admits a non-elementary 
isometric action on a Hadamard manifold $M$ of bounded negative curvature.
Then there exists
 a number $p>1$ and a representation of $\Gamma$ on an $L^p$-space
$V$ with $H^1(\Gamma,V)\not=0$.
\end{theo}

This result is likely to be far from optimal since the requirement that
$\Gamma$ acts on a smooth manifold of bounded negative curvature 
rather than on an arbitrary Gromov hyperbolic geodesic metric space $Y$ 
is very strong. However, 
some assumptions on $Y$ are necessary for the statement of the 
theorem to hold true. Namely, there are finitely generated groups $\Gamma$ which admit
acylindrical and hence non-elementary 
actions on some hyperbolic geodesic metric space, but such that  
for any $p>1$, any affine isometric action of $\Gamma$ 
on an $L^p$-space has a fixed point 
\cite{MO19}. We refer to the very recent
article \cite{DMcK23} for a comprehensive
discussion of related results. For finitely generated groups acting 
properly discontinuously, a version of the Theorem  is contained in 
\cite{BMV05}.

\begin{cor}\label{rankone}
  Let $\Gamma$ be a countable subgroup of a simple rank one Lie group $G$.
  If $\Gamma$ is not contained in a compact or parabolic subgroup of $G$ then
  there exists a representation of $\Gamma$ on an $L^p$-space
 $V$ with $H^1(\Gamma,V)\not=0$. 
\end{cor}

If $G$ does not have property (T), that is, if $G=SO(n,1)$ or $G=SU(n,1)$, 
then this result is well known, and we can in fact choose $p=2$
(Theorem 2.7.2 of \cite{BHV08}). 

%

The proof of Theorem \ref{main} uses an idea due to Nica
\cite{Ni13}. Namely, let $M$ be a 
Hadamard manifold of bounded negative curvature,
with ideal boundary $\partial M$. 
The geodesic flow $\Phi^t$ acts on the unit tangent bundle 
$T^1M$ of $M$ preserving the 
\emph{Lebesgue Liouville measure} 
$\lambda$. This measure disintegrates to a 
Radon measure $\hat \lambda$ on the space of geodesics 
$\partial M\times \partial M-\Delta$ which is invariant under the 
action of the isometry group ${\rm Iso}(M)$ of $M$.
In particular, for any $p\geq 1$, ${\rm Iso}(M)$ 
acts isometrically on 
$L^p(\partial M\times \partial M-\Delta,\hat \lambda)$.
For sufficiently large $p$ we construct a cocycle for this action 
and show that its restriction to the subgroup $\Gamma$ 
 is unbounded provided that the action of
$\Gamma$ is non-elementary. 

The organization of this article is as follows. In Section \ref{coho} we study
actions of a 
group $\Gamma$ on compact metric measure spaces and formulate a condition 
for such an action which 
is sufficient for the construction of a cocycle with values in an $L^p$-space.
In Section \ref{products} we impose some further constraints which guarantee
that the cocycle yields a nontrivial cohomology class.
In Section \ref{anahlfors} we construct an Ahlfors regular
distance function on the ideal boundary $\partial M$ of
a Hadamard manifold $M$ of bounded negative curvature with 
the additional assumption that the covariant derivative of the curvature
tensor is uniformly bounded in norm. 
The distance function $d$ will in general not be
a Gromov metric on $\partial M$, but it is contained in its coarse conformal
gauge. In Section \ref{groupsof} we verify that the conditions
formulated in Section \ref{coho} and Section \ref{products} are
fulfilled for the action of the isometry group of $M$ on
$(\partial M, d)$. This yields the proof of Theorem \ref{main} under the 
additional assumption of bounded covariant derivative of the curvature 
tensor. Smoothing a given metric with the Ricci flow \cite{K05} removes this
additional assumption. 

The appendix contains some regularity result
for the shape operator of horospheres of a 
Hadamard manifold of bounded negative curvature
and bounded covariant derivative of the curvature tensor which 
is used in an essential way in the construction of the Ahlfors
regular metric $d$ on $\partial M$ and which 
we were unable to locate in the
literature.

{\bf Acknowledgement:} I thank Cornelia Drutu and John McKay for useful
  discussions.

\section{Actions on compact metric measure spaces}
\label{coho}

In this section $X$ denotes a compact Hausdorff space. 

\begin{definition}\label{width}
A \emph{width} on $X$ is a continuous symmetric function
$\iota:X\times X\to [0,\infty)$ with $\iota(x,x)=0$ for all $x\in X$.
We call a space $X$ equipped with a width a \emph{width space}.
\end{definition}

\begin{example}
If $X$ is metrizable then 
a metric $d$ on $X$ defining the given topology is an example of 
a width, and the same holds true for the trivial function $X\times X\to \{0\}$. 
If $\iota$ is a width and if $\alpha>0$ is arbitrary, then 
$\iota^\alpha$ is a width. 
\end{example}

A width is not required to satisfy the triangle inequality. Note however
that by continuity of $\iota$ and compactness of $X$ and hence of 
$X\times X$, a width is a 
bounded function.

\begin{definition}\label{width2}
Let $(X,\iota),(Y,d)$ be two width spaces. A map 
$F:X\to Y$ is \emph{Lipschitz continuous} if there exists a number 
$L>0$ such that $d(Fx,Fy)\leq L\iota (x,y)$ for all $x,y$. 
\end{definition}

If $F:X\to Y$ is Lipschitz, and if $x,y\in X$ are such that
$\iota(x,y)=0$, then $d(Fx,Fy)=0$. 

Most width spaces admit very few Lipschitz functions. 

\begin{example} Let $I\subset \mathbb{R}$ be the unit interval, equipped with the
  standard metric $d$. Then the space $(I,d^{2})$
 is a 
width space for which 
any Lipschitz function 
$f:(I, d^{2})\to \mathbb{R}$ is constant.  This  can be seen by noting that 
such a Lipschitz function is differentiable everywhere,
with vanishing differential. 
In other words, raising a metric to a power bigger than one
destroys Lipschitz functions but also the triangle inequality. Note however
that for $\alpha\in (0,1)$ the snowflake $(I,d^\alpha)$ is a metric space.
\end{example} 

Let $\Gamma$ be a countable group which acts
as a group of 
homeomorphisms on  the width space 
$(X,\iota)$. We assume that the diagonal action of $\Gamma$ on 
$X\times X$ preserves the fat diagonal
$\Delta=\{(x,y)\in X\mid \iota(x,y)=0\}$ in $X\times X$. Note that $\Delta$ is 
a closed subset of $X\times X$ by continuity of $\iota$.

Assume that 
$\Gamma$ preserves the measure class of a Borel probability measure $\mu$
on $X$ 
of full support without atoms. Then for all $\phi\in \Gamma$ and for
$\mu$-almost every $x\in X$, the Jacobian of $\phi$ is defined at $x$.
Here by Jacobian we denote the Radon Nikodym derivative
of $\phi_*\mu$ with respect to $\mu$. In the sequel we also speak
about the Radon Nikodym derivative of $\phi$, and we denote it by
${\rm RN}(\phi)$.

Following \cite{Ni13}, for a number $Q>0$ define 
\[\nu=\iota^{-2Q}\mu\times \mu,\]
which is thought of as a measure on $X\times X-\Delta$.
The following is fairly immediate from the above discussion.

\begin{lemma}\label{almostinvariant}
The measure $\nu$ on $X\times X-\Delta$ 
is locally finite and 
quasi-invariant under the diagonal action of
$\Gamma$. The Radon Nikodym derivative ${\rm RN}_\nu(\phi)$ 
of $\phi$ 
with respect to $\nu$ 
equals
\begin{equation}\label{rn}
{\rm RN}_\nu(\phi)(x,y)=
\iota(\phi(x),\phi(y))^{-2Q} {\rm RN}(\phi)(x){\rm RN}(\phi)(y)\iota(x,y)^{2Q}.
 \end{equation}
\end{lemma}
\begin{proof} Since $ \mu$ is quasi-invariant under the action of 
$\Gamma$, and $\Gamma$ preserves the fat diagonal, 
the same holds true for $\nu$. 
Let $\phi\in \Gamma$ and let $(x,y)\in X\times X-\Delta$. It is straightforward that
the formula (\ref{rn}) computes the 
Radon Nikodym derivative ${\rm RN}_\nu(\phi)(x,y)$ 
of $\phi$ at $(x,y)$ with respect to the measure 
$\nu$. 

That the measure $\nu$ on $X\times X-\Delta$ 
is locally finite is immediate from continuity
of the width $\iota$ and finiteness of $\mu$.
\end{proof}

 
 We now formulate a condition for a measure class preserving 
 action of a group $\Gamma$ on $(X,\iota,\mu)$ which 
ensures that there exists a cocycle for the action
with values in $L^p(\nu)$.

{\bf Condition $(*)$:} The group $\Gamma$ acts by bi-Lipschitz transformations
on $(X,\iota)$, and 
for every $\phi\in \Gamma$, the 
Radon Nikodym derivative ${\rm RN}(\phi)$ of $\phi$ is a Lipschitz
continuous function $(X,\iota)\to (0,\infty)$ which is bounded away from zero. 

Note that under Condition $(*)$, by compactness of $X$ and continuity,
for each $\phi\in \Gamma$ the Radon Nikodym derivative 
${\rm RN}(\phi)$ of $\phi_*\mu$ with respect to $\mu$ 
is a bounded function on $X$. 
Moreover, since for each $\phi$ there exists a number $L>0$ (the 
Lipschitz constant of $\phi^{-1}$) such that
$\iota(x,y)\leq L\iota (\phi(x),\phi(y))$ for all $(x,y)\in X\times X-\Delta$, 
Lemma \ref{almostinvariant} shows 
that the function 
${\rm RN}_\nu(\phi)$ is bounded as well. 

%


As a consequence, for each $p>1$ we obtain
a representation of
$\Gamma$ on the space 
\[L^p(\nu)=\{f: X\times X-\Delta \to \mathbb{R}, 
\int \vert f\vert^p d\nu<\infty\}\] 
by 
\begin{equation}\label{lpaction}
  (\phi f)(x,y)=f(\phi^{-1}(x),\phi^{-1}(y)).\end{equation}
Namely, since for each $\phi\in \Gamma$ the function
${\rm RN}_\nu(\phi)$ is pointwise uniformly bounded,  
via the 
formula (\ref{lpaction}),  each $\phi\in \Gamma$ 
acts as a bounded linear operator on $L^p(\nu)$
for every $p\geq 1$.
The action is isometric if and only if the action of 
$\Gamma$ on $(X\times X-\Delta,\nu)$ is measure preserving.

Following Bourdon and Pajot \cite{BP03}, 
for $p\geq 2Q$ define the \emph{Besov space} ${\mathfrak B}_p(X)$ to consist of 
all measurable functions $f:X\to \mathbb{R}$ for which the 
\emph{Besov semi-norm} 
\[\Vert f\Vert_{{\mathfrak B}_p}=
\bigl(\int \int \vert f(x)-f(y)\vert^p \iota^{-2Q}(x,y)d\mu(x)d\mu(y)\bigr)^{1/p}\]
is finite.

\begin{lemma}\label{embed}
For each $p\geq 2Q$, 
the space of Lipschitz function $(X,\iota)\to \mathbb{R}$ embeds into
$\mathfrak{B}_p$.
\end{lemma}  
\begin{proof}
  Let $f:(X,\iota)\to \mathbb{R}$ be Lipschitz continuous. Then
there exists a number $L>0$ with 
 $\vert f(x)-f(y)\vert\leq L\iota(x,y)$ for all $x,y$.

 By continuity of $\iota$, for all $\epsilon >0$ 
the set 
\[D(\epsilon)=\{(x,y)\in X\times X\mid
\iota(x,y)\geq \epsilon\}\] 
is a compact subset of $X\times X-\Delta$, and 
$D(\epsilon)\subset D(\delta)$ for $\epsilon >\delta$, 
$\cup_{\epsilon >0}D(\epsilon)=X\times X-\Delta$. 
For all $\epsilon >0$ we have
\[\int_{D(\epsilon)}\vert f(x)-f(y)\vert^p \iota^{-2Q}(x,y)d\mu(x)d\mu(y)
\leq L^p \int_{D(\epsilon)} \iota(x,y)^{p-2Q}d\mu(x)d\mu(y)\leq C\]
for a universal constant $C>0$ since $p\geq 2Q$ by assumption, since 
$\iota$ is a bounded function and $\mu\times \mu$ is a probability measure.

The statement now follows from Lebesgue's dominated convergence theorem, 
applied to the functions
\[F_\epsilon(x,y)=\begin{cases} \vert f(x)-f(y)\vert^p \text{ if } (x,y)\in D(\epsilon) \\
0 \text{ otherwise }
\end{cases}\]
and the measure $\nu$.
\end{proof}

Via the map which associates to $f\in \mathfrak{B}_p$ the
function $\Psi(f)(x,y)=f(x)-f(y)$,  
the Besov space ${\mathfrak B}_p$ embeds into 
$L^p(X\times X-\Delta,\nu)$.

For $\phi\in \Gamma$ and $(x,y)\in X\times X$ define
\begin{equation}\label{cocycle}
  c_\phi(x,y)=\log {\rm RN}(\phi)(x)-\log {\rm RN}(\phi)(y).
\end{equation}
Then $c_\phi$ is a measurable \emph{cocycle} for the 
diagonal action of $\Gamma$ on 
$X\times X$.
This means that 
\[c_{\phi \circ \psi}(x,y)=c_{\phi}(\psi(x,y))+c_{\psi}(x,y)\]
for all $(x,y)\in X\times X$ and all $\phi,\psi\in \Gamma$.

The following lemma explains the significance of Condition $(*)$.

 \begin{lemma}\label{integrable}
Assume condition $(*)$.
Then 
for $p\geq 2Q$ and each $\phi\in \Gamma$, we have
 $\log {\rm RN}(\phi)\in {\mathfrak B}_p$. In particular,
   the cocycle $c_\phi$ on $X\times X-\Delta$ consists of $L^p$-integrable
functions  with respect to $\nu$.
 \end{lemma}
\begin{proof} By assumption, the function 
${\rm RN}(\phi)$ assumes values in a compact interval $[a,b]\subset (0,\infty)$,
moreover as a function
$(X,\iota)\to [a,b]$, it is Lipschitz continuous. 

Now the restriction of the function $\log$ to a compact interval 
$[a,b]\subset (0,\infty)$ is Lipschitz continuous and hence the same holds true
for the composition $x\to \log({\rm RN}(\phi(x)))$ since the composition of 
Lipschitz functions is Lipschitz. Thus the lemma now follows from
Lemma \ref{embed}.
\end{proof}


\begin{remark} 
By the formula (\ref{rn}), if Condition $(*)$ holds true 
then for 
$p\geq 2Q$ the group $\Gamma$ admits
an isometric action on $L^p(\nu)$ via 
\[(\phi,f)(x,y)=f(\phi^{-1}(x),\phi^{-1}(y)){\rm RN}_\nu(\phi)^{-1/p}(x,y).\]
In other words, we obtain a representation $\Pi$ of $\Gamma$ into the group of linear
isometries of $L^p(\nu)$. 

A cocycle for this isometric action is a function $a:\Gamma\to L^p(\nu)$ such that
\[a(\phi \circ \psi)= \Pi(\phi)(a(\psi)) +a(\phi).\]
Thus the cocycle $c$ is a cocycle for this
isometric representation if and only if 
$\nu$ is invariant under 
the action of $\Gamma$. A necessary condition for this to hold is 
that the functions ${\rm RN}_\nu(\phi)$ are uniformly bounded, 
independent of $\phi$. In general, it is unclear whether there exists
a constant $Q>0$ such that this boundedness condition holds true.
\end{remark}

\begin{example}\label{projective}
  Consider the standard projective action of the group $\Gamma=SL(n,\mathbb{Z})$
  on $X=\mathbb{R}P^{n-1}$ $(n\geq 2)$. This action is by diffeomorphisms and hence if 
 $\mu$ denotes the volume form induced by the round metric, then 
  the Radon Nikodym derivatives of the elements of $\Gamma$ 
  are Lipschitz continuous.
  Thus for any $Q\geq 1$ and any $p\geq 2Q$ 
  one obtains a cocycle $c$ with values in the space of 
  functions which are composed of pull-backs of 
  functions in the
  Besov space $\mathfrak{B}_p$ via the first and second factor projection.  
  For $n=2$ and $Q=1,p\geq  2$, this 
  is a cocycle for an isometric action on an $L^p$-space which 
  defines a nontrivial cohomology class for $\Gamma$. For $n\geq 3$ 
  it is unclear whether there exists $Q\geq 1,p\geq 2Q$ 
  such that this cocycle is a cocycle for an isometric action on 
$L^p$.
By the main result of \cite{BFGM07}, if such numbers $p,Q$ exist
then the cocycle is a coboundary.  
\end{example}

\section{Measure preserving actions on products}\label{products} 

The goal of this section is to find conditions which 
guarantee that the cocycle for the group $\Gamma$ 
constructed in Section \ref{coho} 
defines a nontrivial class in the first cohomology of $\Gamma$
with coefficients in the representation.

Assume for the remainder of this section that $(X,d,\mu)$ is 
a compact  
\emph{Ahlfors regular} metric measure space of dimension $Q\geq 1$.
This means that $(X,d)$ is a metric space, and 
there exists a number $C>0$ such that 
\[\mu(B(\xi,r))\in [C^{-1}r^Q,Cr^Q]\]
for all $\xi\in X$ and all $r\leq {\rm diam}(X)/2$, where 
$B(\xi,r)$ denotes the open ball of radius $r$ about $\xi$.
Assume furthermore that the countable group $\Gamma$ acts on 
$(X,d)$ as a group of bi-Lipschitz homeomorphisms.
This implies that $\Gamma$ preserves the measure class of $\mu$.
In particular, for every $\phi\in \Gamma$ and $\mu$-almost every
$x\in X$ the Radon Nikodym dervative ${\rm RN}(\phi)(x)$ of 
$\phi_*\mu$ with respect to $\mu$ exists at $x$.
Put
\[\langle \phi,x\rangle=\log 
{\rm RN}(\phi)(x);\] 
then $c_\phi(x,y)=\langle\phi,x\rangle-\langle \phi, y\rangle$.

A homeomorphism $\phi$ of a compact metric space $(X,d)$ has \emph{an attracting 
fixed point} $x$ if there exists a compact neighborhood $U$ of $x$
such that $\phi(U)\subset U$ and 
\[\cap_{j>0} \phi^j(U)=\{x\}.\]

Write as before 
$\nu=d^{-2Q} \mu\times \mu$, viewed as a measure on 
$X\times X-\Delta$.
For a measure class preserving map 
$\phi$ put ${\rm RN}_\nu(\phi)=d\phi_*\nu/d\nu$.

 
 
\begin{proposition}\label{lowerestimate}
Let $(X,d,\mu)$ be an Ahlfors regular metric measure space 
of dimension $Q\geq 1$.
Assume that $\Gamma$ acts on $X$ preserving the measure class of 
$\mu$.
Assume moreover that  there is $\phi\in \Gamma$ and a number $C>0$ 
with the following properties. 
\begin{enumerate}
\item $\phi^{-1}$ has an attracting fixed point.
\item For all $\ell \in \mathbb{Z}$ we have ${\rm RN}_\nu(\phi^\ell)\in 
[C^{-1},C]$. 
\end{enumerate}
Then for $p\geq 2Q$ and $m>0$ 
there is a number $k_0=k_0(\phi,p,m)>0$ such that
\[\Vert c_{\phi^k}\Vert_{L^p}^p\geq m\]
for all $k\geq k_0$.
\end{proposition}
\begin{proof} 
%
%
Let $\phi\in \Gamma$ and assume that $\phi^{-1}$ admits an attracting fixed point 
$z\in X$. 
Assume moreover that there exists a number $C>0$ such 
if we denote
by ${\rm RN}(\phi^\ell)$  the Radon Nikodym derivative of 
$\phi^\ell$ with respect to $\mu$ then we have 
\begin{equation}\label{rn2}
{\rm RN}_\nu(\phi^\ell)(x,y)={\rm RN}(\phi^\ell)(x){\rm RN}(\phi^\ell)(y) d(\phi^\ell(x),
\phi^\ell(y))^{-2Q}d(x,y)^{2Q}\in [C^{-1},C]\end{equation}
for all $\ell\in \mathbb{Z}$ and 
almost all $(x,y)\in X\times X-\Delta$.
By enlarging $C$ if necessary we may moreover assume that 
for all $x\in X$ and all $r<{\rm diam}(X)$ we have 
\[\mu(B(x,r))\in [C^{-1}r^Q,Cr^Q].\]
 
 Let $b>0$ be such that the closed ball 
 $\bar B(z,b)$ of radius $b$ about $z$ is contained in a compact neighborhood
 $U$ of $z$ as
 in the definition of an attracting fixed point. 
 By assumption, there exists a number $\ell_0 >0$ such that 
$\phi^{-\ell_0}(U)\subset \bar B(z,b)$ and hence 
\[\phi^{-\ell} \bar B(z,b)\subset
\phi^{-\ell_0}(\phi^{-\ell+\ell_0}U)\subset \phi^{-\ell_0}U\subset \bar B(z,b)\text{ for all }\ell\geq \ell_0.\]
 %
The same argument also shows that if $\phi^{-j}B(z,b)\subset B(z,\delta)$ for some
$\delta >0$, then $\phi^{-k}B(z,b)\subset B(z,\delta)$ for all $k\geq \ell_0+j$. 


Let $\tau=\max\{{\rm diam}(X)/b,(2C^2)^{1/Q}\}>2$. 
Then for all $r\leq b$ we have
\begin{equation}\label{volume}
\mu(B(z,r)-B(z,\tau^{-1}r))\geq C^{-1}r^Q-C\tau^{-Q}r^Q
\geq C^{-1}r^Q/2.\end{equation}
Choose a sequence $j_m\to \infty$ such that 
\begin{equation}\label{contraction}
\phi^{-j_m} B(z,b)\subset B(z,b\tau^{-8m})\end{equation}
for all $m\geq 1$. Note that we then have 
$\phi^{-k}B(z,b)\subset B(z,b\tau^{-8m})$ for any $k\geq j_m+\ell_0$ and 
any $m$.

Let $m\geq 1$ be arbitrary. 
We claim that there is a universal constant $u>0$ such that for 
all $\ell \leq m-2$, all $x\in 
B(z,b\tau^{-8\ell})-B(z,b\tau^{-8\ell-2})$  
and all $y\in B(z,b\tau^{-8\ell-6})-B(z,b\tau^{-8\ell-8})$ we have
$\langle \phi^{j_m},y \rangle-\langle \phi^{j_m},x \rangle \geq u$. 

To see that this is the case, 
note that since $\ell\leq m-2$, 
we have 
\[\phi^{j_m}x,\phi^{j_m} y\in X-B(z,b)\] 
and hence
$d(z,\phi^{j_m} x)\in [b,\tau b], d(z,\phi^{j_m}y)\in [b,\tau b]$ 
(recall that ${\rm diam}(X)\leq \tau b$). 
Now let us assume for the moment that the Radon Nikodym derivative of 
$\phi^{j_m}$ exists a $z$. Then together with 
the assumption (\ref{rn2}) and the fact that $\phi^{j_m}z=z$ for all $m$, we conclude 
that 
\begin{align} \label{ratio}
 &\frac{ {\rm RN}(\phi^{j_m})(z){\rm RN}(\phi^{j_m})(y) d(\phi^{j_m}(z),
\phi^{j_m}(y))^{-2Q}d(z,y)^{2Q}}
{{\rm RN}(\phi^{j_m})(z){\rm RN}(\phi^{j_m})(x) d(\phi^{j_m}(z),
\phi^{j_m}(x))^{-2Q}d(z,x)^{2Q}}   \\
\sim & \frac{ {\rm RN}(\phi^{j_m})(y) d(z,y)^{2Q}}
{ {\rm RN}(\phi^{j_m})(x) d(z,x)^{2Q}} \sim {\rm const}. \notag
\end{align}
Here the first $\sim$ in the estimate means equality up to a factor contained
in the interval $[\tau^{-4Q},\tau^{4Q}]$ (by erasing two factors in the numerator
and denominator using the above estimate), 
and the
second $\sim$ means that value of the ratio is contained in 
the interval $[C^{-2}\tau^{-4Q}, C^2\tau ^{4Q}]$ (which follows from the fact that
by (\ref{rn2}), the ratio of numerator and denominator in the 
first term of the expression (\ref{ratio}) is at most $C^2$).
Now $\frac{d(z,y)^{2Q}}{d(z,x)^{2Q}}\leq \tau^{-8Q}$, and as
$\tau^{-8Q}C^2\tau^{4Q}\leq \tau^{-3Q}$, 
taking the logarithm yields the claim. 

The Radon Nikodym derivative of $\phi^{j_m}$ with respect to $\mu$ 
may not exist at $z$. However, 
it exists almost everywhere 
and hence by Ahlfors regularity of $\mu$, we can 
find a point $\hat z$ arbitrarily close to $z$ which is mapped by
$\phi^{j_m}$ into an arbitrarily small neighborhood of $z$ and such that 
the Radon Nikodym derivative of $\phi^{j_m}$ exists at $\hat z$. Replacing 
$z$ by $\hat z$ in formula (\ref{ratio}) then yields the statement we were looking for.

On the other hand, by (\ref{volume}), and Ahlfors regularity of $\mu$, 
there exists a universal constant
$\kappa >0$ such that 
\[\nu((B(\xi,b\tau^{-8\ell})-B(\xi,b\tau^{-8\ell -2}))
\times (B(\xi,b\tau^{8\ell-6})-B(\xi,b\tau^{8\ell-8}))) \geq \kappa\] for all 
$\ell\in [2,m]$. Since the value of $c_{\phi^{j_m}}$ on these sets is 
bounded from below by $u>0$, we have 
\[\int \vert c_{\phi^{j_m}}\vert^p d\nu\geq (m-2)u^p\kappa.\]
As the right hand side of this inequality tends to $\infty$ as 
$m\to \infty$, the proposition follows. 
\end{proof}

\section{An Ahlfors regular metric on the boundary of a Hadamard manifold}
\label{anahlfors}

In this section we consider an $n$-dimensional 
simply connected complete Riemannian manifold 
$(n\geq 2)$ of 
sectional curvature contained in the interval $[-b^2,-a^2]$ for numbers
$0<a<b<\infty$. 
We also require that there is a universal
upper bound for the norm of the covariant derivative 
$\nabla R$ of the curvature tensor $R$ of $M$.
Our goal is to construct a metric $d$ on 
the \emph{ideal boundary} $\partial M$ of $M$
with the following properties.
\begin{enumerate}
\item $d$ is Ahlfors regular.
\item The Ahlfors regular measure $\mu$  defined by $d$ is 
contained in the Lebesgue measure class.
\item Isometries of $M$ 
act by bi-Lipschitz transformations on $(\partial M,d)$.
\end{enumerate}
Note that for general Hadamard manifolds
of bounded negative curvature, it is not clear whether (2) makes sense, 
and it is to this end that we shall use the assumption 
that $\vert \nabla R\vert$ is bounded.

A point $\xi\in \partial M$ determines a \emph{Busemann function}
$b_\xi$ at $\xi$. Its level sets are the \emph{horospheres} at 
$\xi$. By the assumption on $M$, 
these Busemann functions 
are of class $C^3$, and their gradients ${\rm grad}\,b_\xi$ 
are $C^2$-vector fields on 
$M$ \cite{Shc83}.

Let $T^1M$ be the unit tangent bundle of $M$. 
The metric on $M$ induces a natural metric on $T^1M$, the 
\emph{Sasaki metric}. This metric defines a distance function
and hence a H\"older structure for functions on $T^1M$. 
The canonical projection 
\[\Pi:T^1M\to M\]
is a Riemannian submersion.

For a point $x\in M$ and a unit tangent vector 
$v\in T_x^1M$ let $m(v)>0$ be the mean curvature at $x$ of the 
horosphere in $M$ whose outer normal field passes through $v$.
That is, the horosphere is a level set of the Busemann function defined 
by the ideal boundary point $\gamma_v(-\infty)\in \partial M$, where
$\gamma_v$ denotes the geodesic with initial velocity $\gamma_v^\prime(0)=v$.

The following result is perhaps well known. We provide a proof in the appendix
(Corollary \ref{meancurvature}). 

\begin{proposition}\label{hoelder}
The function $m:v\to m(v)$ on $T^1M$ is H\"older continuous.
\end{proposition}

We next observe

\begin{lemma}\label{symmetry}
There exists a number $C_0>0$ with the following property. 
Let $\gamma\subset M$ be any geodesic and let $t>0$; then 
\[\vert \int_0^t m(\gamma^\prime(s))ds-\int_0^t m(-\gamma^\prime(s))ds \vert 
\leq C_0.\]
\end{lemma}
\begin{proof}
The Lebesgue Liouville measure  
$\lambda$ on $T^1M$ is the measure 
defined by the volume form of the Sasaki metric. This volume form,
again denoted by $\lambda$, is
invariant under the geodesic flow $\Phi^t:v\to \gamma_v^\prime(t)$. It can be 
described as follows.

The tangent bundle $TT^1M$ of $T^1M$ has an orthogonal decomposition 
as $TT^1M={\mathcal H}\oplus {\mathcal V}$ where 
${\mathcal V}$ is the \emph{vertical tangent bundle}, that is,
the tangent bundle of the 
fibers of the fibration $T^1M\to M$, and where
${\mathcal H}$ is the \emph{horizontal bundle}
defined by the Levi Civita connection. 
Then 
\[\lambda=\omega_{\mathcal H}\wedge \omega_{\mathcal V}\] where 
$\omega_{\mathcal H}$ is an $n$-form which 
annihilates ${\mathcal V}$, 
$\omega_{\cal V}$ is an $(n-1)$-form which 
annihilates ${\cal H}$ and such that 
$\omega_{\cal V}$, $\omega_{\cal H}$ are defined by a choice of an orientation
and the Riemannian metric on ${\cal H},{\cal V}$.

The contraction $\iota_X\lambda$ of $\lambda$ 
with the generator $X$ of the geodesic 
flow $\Phi^t$ is a smooth $(2n-2)$-form $\omega$  
which is invariant under 
$\Phi^t$ and which equals $\iota_X\omega_{\cal H}\wedge
\omega_{\cal V}$. This $(2n-2)$-form then defines a
Radon measure on the space of geodesics 
$\partial M\times \partial M-\Delta$ 
which is invariant under the action of ${\rm Iso}(M)$.

There exists another natural $2n-2$-form on $T^1M$ which is defined 
as follows. 
For a given vector $v\in T^1M$ we can consider the
submanifolds $W^{ss}(v),W^{su}(v)$ 
of $T^1M$ defined by the inner and outer normal field of the
horosphere through $\Pi(v)$ at $\gamma_v(\infty)$ and $\gamma_v(-\infty)$, respectively.
By the assumption on $M$, this construction defines two continuous foliations 
$W^{ss},W^{su}$ of $T^1M$, with leaves 
of class $C^2$. Thus the tangent bundles $TW^{ss},TW^{su}$ 
of these foliations are defined, 
and by Proposition \ref{hoelder2}, 
they are H\"older continuous subbundles of $TT^1M$ of dimension 
$n-1$ which do not intersect. 

Namely, for $v\in T^1M$ the orthogonal complement 
$v^\perp$ of $v$ in $T_{\Pi(v)}M$ 
has a natural isometric identification with both the orthogonal
complement of $X$ in ${\cal H}_v$ and the fiber ${\cal V}_v$.  
A tangent vector $Y$ of $TW^{su}$ 
at $v\in T^1M$ 
decomposes into 
\[Y=Y^h+Y^v\text{ where }Y^h\in X^\perp \subset 
{\cal H}_v,Y^v\in {\cal V}_v.\] 
With respect to the natural isometric identification
of $X^\perp\subset {\cal H}_v$ and ${\cal V}_v$, 
the linear map which sends $Y^h$ to $Y^v$ 
is given by the shape operator for the outer normal field
of the 
horosphere defined by $\gamma_v(-\infty)$, 
and this shape operator is a symmetric  
linear operator whose eigenvalues are bounded from above and 
below by universal positive constants. 

Similarly, a tangent vector $Z$ of $TW^{ss}$ at 
$v$ decomposes as 
\[Z=Z^h+Z^v\text{ where }Z^h\in {\cal H}_v,
Z^v\in {\cal V}_v\] 
and the linear map $X^\perp \subset {\cal H}_v\to {\cal V}_v$ which 
sends $Z^h$ to $Z^v$ is the shape operator for the inner normal field
of the horosphere 
defined by $\gamma_v(\infty)$, 
and this shape operator is a symmetric linear operator 
whose eigenvalues are bounded from above and below by universal 
negative constants. Using Proposition \ref{hoelder2}, 
this shows that the bundles $TW^{ss},TW^{su}$ are H\"older continuous and furthermore, 
the angle with respect to the Sasaki metric between a nonzero vector
of $TW^{ss}$ and a nonzero vector of $TW^{su}$ over a point $v\in T^1M$ 
is bounded from below by a
universal positive constant not depending on $v$. 

Thus we obtain another continuous
$(2n-2)$-form $\tilde \omega$ on $T^1M$ by defining  
\[\tilde \omega= \omega^{su}\wedge \omega^{ss}\] where the $n-1$-form
$\omega^{ss}$ annihilates $TW^{su}\oplus \mathbb{R}X$ 
and restricts to the volume form on 
the leaves of the foliation $W^{ss}$ which is induced from the pull-back
of the metric on horospheres in $M$, and similarly for $\omega^{su}$.
The $(2n-2)$-form $\tilde \omega$ annihilates the generator of the geodesic
flow and hence it can be represented as $\iota_X\tilde \lambda$ for 
a $2n-1$-form $\tilde \lambda$. Then $\tilde \lambda=\kappa \lambda$
for a continuous function $\kappa:T^1M\to \mathbb{R}$. By the above discussion,
there exists a constant $C>0$ such that $\kappa(T^1M)\subset [C^{-1},C]$. 

The volume form $\tilde \lambda$ is in general not invariant under 
the geodesic flow $\Phi^t$, but it is quasi-invariant, and its Lie derivative
${\cal L}_X\tilde \lambda$ in direction of the generator $X$ of $\Phi^t$ 
equals
\[{\cal L}_X\tilde \lambda(v)=(m(v)-m(-v))\tilde \lambda\]
since by construction and the properties of the mean curvature of 
horospheres we have
\[(\frac{d}{dt}\omega^{su}\circ \Phi^t\vert_{t=0})(v)=m(v)\omega^{su}(v)
\text{ and } (\frac{d}{dt} \omega^{ss}\circ \Phi^t\vert_{t=0})(v)=-m(-v)\omega^{ss}(v).\]

As a consequence, 
for $v\in T^1M$ and $t>0$ the logarithm at $v$ of the Jacobian of $\Phi^t$  
with respect to the measure $\tilde \lambda$
can be represented as 
\begin{equation} \label{integralbound}
\int_0^t m(\Phi^sv)ds -\int_0^t m(- \Phi^sv)ds.\end{equation}
On the other hand, the measure $\lambda$ is invariant under $\Phi^t$, and 
it equals a uniformly bounded multiple of the measure $\tilde \lambda$. 
Thus the Jacobian of $\Phi^t$ with respect to $\tilde \lambda$ is uniformly bounded,
independent of the basepoint and $t$. Then the same holds true for the 
integral (\ref{integralbound}) which shows 
the lemma. 
\end{proof}

For $v\in T^1M$ define 
\[f(v)=\frac{1}{2}(m(v)+m(-v)).\]
Since the function $m$ is H\"older continuous, the same holds true
for the function $f$. Furthermore, $f$ takes values in 
a fixed interval $[c,d]$ for $0<c<d$ since the mean curvature
of horospheres is bounded from above and below by a positive 
constant only depending on the curvature bounds.

We now use a construction which is well known in the case that the manifold
$M$ is the universal covering of a closed manifold, see for example 
the articles \cite{L95} and \cite{H97}. 

For $x\in M$ we shall define a 
Gromov type product $( \mid )_x$  
on $\partial M$  based at $x$.
To this end
consider first a point $\xi\in \partial M$ and two points 
$x,y\in M$. There are unique geodesic rays 
$\gamma,\eta:[0,\infty)\to M$ connecting $x,y$ to $\xi$, that is, such that
$\gamma(0)=x,\gamma(\infty)=\xi$ and $\eta(0)=y,\eta(\infty)=\xi$.
Given the parameterization for
$\gamma$, there exists a unique 
preferred parameterization $\hat \eta$ of $\eta$ (extended to a geodesic line) 
with $\hat \eta(u)=\eta(0)$ for some $u\in \mathbb{R}$ 
and so that $\gamma, \hat \eta$ are \emph{strongly asymptotic}, that is, 
\[\lim_{t\to \infty} d(\gamma(t),\hat \eta(t))=0.\]
We have

\begin{lemma}\label{converge}
The limit 
\[ q_\xi(x,y)= 
\lim_{t\to \infty}
\bigl(\int_0^t f(\gamma^\prime(s))ds-\int_u^t f(\hat \eta^\prime(s))ds\bigr)\]
exists. 
\end{lemma}
\begin{proof} Let $d$ be the Sasaki metric on $T^1M$. 
If $-a^2<0$ is an upper curvature bound
for $M$ then we have 
\[d(\gamma^\prime(s),\hat \eta^\prime(s))\leq C_0e^{-as}\] for some $C_0>0$ 
depending on $\gamma,\eta$.

Since by Proposition \ref{hoelder} the function $f$ is H\"older 
continuous, there exist numbers $C_1>0,\alpha >0$ so that 
\[\vert f(\gamma^\prime(s))-f(\hat \eta^\prime(s))\vert \leq 
C_1d(\gamma^\prime(s),\hat \eta^\prime(s))^\alpha\leq C_1C_0^\alpha e^{-\alpha a s}
\text{ for all }s.\]
As the function $s\to e^{-\alpha as}$ is integrable on 
$[0,\infty)$, this yields the existence of the limit
\[\lim_{t\to \infty} \bigl(\int_0^t f(\gamma^\prime(s))ds
-\int_u^t f(\hat \eta^\prime(s))ds \bigr).\]
\end{proof}

Note that $q_\xi(x,y)=-q_\xi(y,x)$.

For $x\in M$ and two points $\xi\not=\eta\in \partial M$ let 
$\rho$ be the geodesic connecting $\xi$ to $\eta$.
Choose a point $y\in \rho$ and define
\[(\xi\mid \eta)_x= \frac{1}{2}
(q_\xi(x,y) + q_\eta(x,y)).\]
Thus $(\xi\mid \eta)_x=(\eta\mid \xi)_x$ for all $\xi,\eta$.
We have

\begin{lemma}\label{indep}
$(\xi\mid \eta)_x$ does not depend on the choice of $y\in \rho$.
\end{lemma}
\begin{proof} Parameterize $\rho$ in such a way that $\rho(0)=y$ and 
$\rho(\infty)=\xi$. 
Replacing $y$ by $\rho(t)$ for some $t>0$ adds the 
integral $\int_0^t f(\rho^\prime(s))ds$ to $q_\xi(x,y)$ 
and adds the integral 
$-\int_0^t f(-\rho^\prime(s))ds$ to $q_\eta(x,y)$. Since $f$ is invariant 
under the flip $v\to -v$, the lemma follows. 
\end{proof}

As a consequence of Lemma \ref{indep},  $(\xi \mid \eta)_x$ only depends on $\xi,\eta,x$.
We use these Gromov type products to define a \emph{cross ratio} $[,,,]_x$ on $\partial M$ by 
\[ [\xi_1,\xi_2,\xi_3,\xi_4]_x=(\xi_1\mid \xi_3)_x +(\xi_2\mid \xi_4)_x-
(\xi_1\mid \xi_4)_x-(\xi_2\mid \xi_3)_x.\]
We have 

\begin{lemma}\label{independence}
$[,,,]_x$ does not depend on $x$.
\end{lemma}
\begin{proof}
Let $y\in M$ be another point and let 
$\xi,\eta\in \partial M$. Then we have
\[(\xi\mid \eta)_y=(\xi \mid \eta)_x +
q_\xi(y,x)+ q_\eta(y,x).\]
This formula yields an expression for 
$[\xi_1,\xi_2,\xi_3,\xi_4]_y-[\xi_1,\xi_2,\xi_3,\xi_4]_x$
as a sum of terms $q_{\xi_i}(y,x)$ $(i=1,\dots,4)$.
Each of these terms 
appears twice in this expression, 
with opposite signs, and hence the terms cancel. 
\end{proof}

By construction, for any isometry $\phi$ of $M$, for 
any quadruple $(\xi_1,\xi_2,\xi_3,\xi_4)$ of distinct points in $\partial M$
and any $x\in M$ we have
\[[\phi(\xi_1),\phi(\xi_2),\phi(\xi_3),\phi(\xi_4)]_{\phi(x)}=
[\xi_1,\xi_2,\xi_3,\xi_4]_x\]
and therefore we obtain

\begin{corollary}\label{moebius}
The isometry group of $M$ preserves $[,,,]_x$.
\end{corollary}

Fix again a point $x\in M$. We have

\begin{lemma}\label{ultrametric}
There exists a number $\kappa >0$ such that the 
product $( \mid)_x$ satisfies the $\kappa$-ultrametric 
inequality 
\[(\xi\mid \eta)_x\geq \min\{(\xi \mid\zeta)_x,(\zeta\mid \eta)_x\}-\kappa.\]
\end{lemma}
\begin{proof}
The proof is standard and we only give a sketch. Namely, consider pairwise
distinct points $\xi,\zeta,\eta\in \partial M$. These points define 
a geodesic triangle which is $\delta_0$-thin for some $\delta_0>0$ 
(a side is contained in the $\delta_0$-neighborhood of the union of the other sides) 
since $M$ is hyperbolic 
in the sense of Gromov. 

As the function $f$ is uniformly H\"older continuous and bounded from above
and below by uniform positive constants, we know that there exists 
a number $m>0$ with the following property. 
Let $y$ be the shortest distance projection of $x$ into the geodesic $\theta$ 
connecting $\xi$ to $\eta$ and let $\rho:[0,T]\to M$ be the geodesic connecting
$x$ to $y$; then 
\[\vert (\xi\mid \eta)_x-\int_0^Tf(\rho^\prime(s))ds\vert \leq m.\]
Namely, for this choice of a point $y\in \theta$, we have
$\vert q_\xi(x,y)-\int_0^T f(\rho^\prime(s))ds\vert \leq m$
for a universal constant $m>0$, and the same estimate also holds
true for $q_\eta(x,y)$.

On the other hand, we also know that $y$ is contained in the $\delta_0$-neighborhood 
of one of the geodesics connecting $\xi$ to $\zeta$ or connecting
$\eta$ to $\zeta$, say the geodesic $\gamma$ connecting $\xi$ to $\zeta$. 
Let $y^\prime\in \gamma$ be such that $d(y,y^\prime)\leq \delta_0$.
Let $z$ be the shortest distance projection of $x$ into $\gamma$.  
Assume without loss of generality that $y^\prime$ is contained in the 
subray of $\gamma$ connecting $z$ to $\xi$. 

Consider the geodesic triangle $T$ with vertices $x,z,y^\prime$. It has a right angle at 
$z$. Let $\rho_1:[0,T_1]\to M,\rho_2:[0,T_2]\to M$ be the sides of $T$ 
with vertex $x$ and second vertex
$z,y^\prime$, respectively, 
Using once more H\"older continuity of $f$ and the fact that $f$ is bounded from
above and below by a universal positive constant, we conclude that 
\[\int_0^{T_1}f(\rho_1^\prime(s))ds \leq \int_0^{T_2}f(\rho_2^\prime(s))ds +\ell\]
for a universal constant $\ell>0$. But this means
\[(\xi\mid\eta)_x-(\xi\mid \zeta)_x\geq -2m-\ell \]
which is what we wanted to show.
\end{proof}

A \emph{quasimetric} on a space $X$ is a symmetric function 
$q:X\times X\to [0,\infty)$ which vanishes only for $x=y$ and satisfies 
for some $K>0$
\[q(x,y)\leq K(q(x,z)+q(z,y)) \text{ for all }x,y,z.\]
As a consequence of Lemma \ref{ultrametric}, if for $\epsilon >0$, 
we define  
\[\delta_{x,\epsilon}(\xi,\eta)=e^{-\epsilon (\xi \mid \eta)_x}\] then we have

\begin{lemma}\label{quasimetric}
$\delta_{x,\epsilon}$ is a quasimetric on $\partial M$.
\end{lemma}

Quasimetrics with multiplicative constant sufficiently close to $1$ 
are known to be bi-Lipschitz equivalent
to metrics and hence we have

\begin{lemma}\label{bilipmetric}
For sufficiently small $\epsilon>0$ the function $\delta_{x,\epsilon}$
is bi-Lipschitz equivalent to a metric $d_x$.
\end{lemma}

Let now $\lambda_x$ be the measure on $\partial M$ which is the image of the
Lebesgue measure on the round sphere $T_xM$ under the natural 
homeomorphism $T_x^1M\to \partial M$. The following is the main result
of this section.

\begin{proposition}\label{regular}
The metric measure space $(\partial M,d_x,\lambda_x)$ is Ahlfors regular of 
dimension $1/\epsilon$.
\end{proposition}
\begin{proof}
Since the diameter of $d_x$ is finite, the total mass of the measure $\lambda_x$ is 
finite, and furthermore $d_x$ is bi-Lipschitz equivalent to $\delta_x=\delta_{x,\epsilon}$,  
it suffices to find numbers $C>0,r_0>0$ such that 
\begin{equation}\label{massestimate}
\lambda_x(B(\xi,r))\in [Cr,C^{-1}r]\end{equation}
for 
all $\xi\in \partial M$ and all $r\leq r_0$, 
where $B(\xi,r)=\{\eta\mid e^{-(\xi\mid \eta)_x}<r\}$.

Let $\exp$ be the exponential map of the Riemannian manifold $M$. 
Using the uniform curvature bound, the map $v\to \exp{ (10v)}$ maps 
the unit sphere $T_x^1M$ onto the distance sphere $S(x,10)$ of radius
$10$ about $x$, and its Jacobian for the standard metric on $T_x^1M$ and 
the metric on $S(x,10)$ induced from the metric on $M$ 
is a smooth function $h$ with values in 
$[\kappa_1,\kappa_2]$ for constants $0<\kappa_1 <\kappa_2$ 
not depending on $x$. Thus it suffices to show the 
estimate (\ref{massestimate}) for the 
measure $h\lambda_x$.

Let $r_0>0$ be sufficiently small that 
$\int_0^{10} m(\gamma^\prime(t))dt \leq -\log r_0$
for every geodesic $\gamma$ in $M$.
We next  claim that for $r\leq r_0$
and $\xi\in \partial M$ the
ball $B(\xi, r)$ 
can be understood as follows.

Let $\gamma_\xi:[0,\infty)\to M$ 
be the geodesic ray connecting $x=\gamma_\xi(0)$ to $\xi$.
Let $R>10$ be such that 
\[\int_0^R m(\gamma^\prime(t))dt =-\log r.\]
As the function $m$ is positive and bounded from below by a positive number, such 
a number $R>10$ exists and is unique.

For $q>0$ consider the ball \[B^S(\gamma_\xi(R),q)\subset S(x,R)\] 
of radius $q$ about 
$\gamma_\xi(R)$ in the distance sphere $S(x,R)$, equipped with the 
intrinsic path metric. Let moreover $A(\xi,R,q)\subset \partial M$
be the set of all endpoints of geodesic rays starting at $x$
which cross through $B^S(\gamma_\xi(R),q)$.
We claim that there exist numbers 
$0< q<u$ not depending on $\xi$ and $r$ 
such that 
$A(\xi,R,q)\subset B(\xi,r)\subset A(\xi,R,u)$.

Before we prove the claim we 
show that it implies the proposition. Namely, 
using the assumption that $R\geq 10$, comparison shows that the volumes 
of the balls $B^S(\gamma_\xi(R),q),B^S(\gamma_\xi(R),u)$ for the induced metric on 
$S(x,R)$ are contained in the interval $[\chi,\chi^{-1}]$ for a universal constant
$\chi >0$. 
Thus to establish the desired lower and upper bound for the $h\lambda_x$-volume of 
$B(\xi,r)$, it suffices to show that there exists a number $C>0$ such that 
for each $y\in B^S(\gamma_\xi(R),u)$ the Jacobian of the 
radial projection $\pi_{R,10}:S(x,R)\to S(x,10)$ at $y$ 
is contained in $[C^{-1}r,Cr]$.

The negative of the logarithm of the Jacobian of the map $\pi_{R,10}$ 
at $y=\gamma_\eta(R)$ for some $\eta\in \partial M$ 
can be computed as an integral
\[\int_{10}^R m^S_t(\gamma_\eta^\prime(t))dt\]
where 
$m^S_t(\gamma_\eta^\prime(t))$ is the mean curvature of the distance sphere
$S(x,t)$ at the point $\gamma_\eta(t)$.  

By Lemma \ref{hoelderclose}, we have
\[\vert m_t^S(\gamma_\eta^\prime(t))-m(\gamma_\eta^\prime(t))\vert \leq e^{-\alpha t}\]
for a universal constant $\alpha >0$ and therefore 
\[\vert \int_{10}^R m_t^S(\gamma_\eta^\prime(t))dt -\int_{10}^R m(\gamma_\eta^\prime(t)) dt \vert \leq 
C_0\]
where $C_0>0$ is a universal constant.
Thus we are left with showing that for each $\eta\in \partial M$ with
$\gamma_\eta(R)\in B^S(\gamma_\xi(R),u)$ we have 
\begin{equation}\label{comparison5} 
\vert \int_{10}^R m(\gamma_\eta^\prime(t))dt -\int_{10}^R 
m(\gamma_\xi^\prime(t))dt \vert \leq C_1\end{equation}
for a universal constant $C_1>0$.

However, comparison shows that for such a point $\eta$ we have
\[d(\gamma_\eta^\prime(t),\gamma_\xi^\prime(t))\leq  C_2e^{a(t-R)} \text{ for all }
10\leq t\leq R\]
for a universal constant $C_2>0$ and consequently the 
estimate (\ref{comparison5}) follows once more 
from H\"older continuity of $m$.

It remains to show the inclusion $A(\xi,R,q)\subset B(\xi,r)\subset A(\xi,R,u)$ for universal constants
$0<q<u$. 
To this end recall from the proof of Lemma \ref{ultrametric} that there exists a number
$z>0$ with the following property. Let $\eta\in \partial M$ and let 
$\gamma_\eta$ be the geodesic ray connecting $x$ to $\eta$. 
Let $\rho$ be the geodesic connecting $\xi$ to $\eta$ and let 
$y$ be the shortest distance projection of $x$ into $\rho$.
Let $\zeta:[0,T]\to M$ be the geodesic connecting $x$ to $y$
and let $\ell=\int_0^T f(\zeta^\prime(t)) dt$. If 
$\ell > -\log R+ z$ then $\eta\in B(\xi,r)$, and if 
$\ell< -\log R-z$ then $\eta\not\in B(\xi,r)$.

The containments then follow from H\"older continuity of the function $f$ and 
Lemma \ref{symmetry}. 
\end{proof}

\begin{corollary}\label{absolute}
For $x,y\in M$, the measures $\lambda_x,\lambda_y$ on 
$\partial M$ are absolutely continuous, with uniformly bounded 
Radon Nikodym derivative.
\end{corollary}
\begin{proof}
By construction, for $x\not=y$ the metrics $d_x,d_y$ on 
$\partial M$ are bi-Lipschitz equivalent. Since 
by Proposition \ref{regular} the
metric measure spaces $(\partial M,d_x,\lambda_x)$ 
and $(\partial M,d_y,\lambda_y)$ are Ahlfors regular, of 
dimension $1/\epsilon$, absolute continuity of the measures 
$\lambda_x,\lambda_y$ is immediate.
\end{proof}

\section{Groups of isometries of Hadamard manifolds}\label{groupsof}

The goal of this section is the proof of the theorem 
from the introduction.  We begin with a group $\Gamma$ which admits a non-elementary
isometric action on a Hadamard manifold $M$ of bounded negative curvature 
and such that the covariant derivative of the curvature  tensor of $M$ is uniformly bounded.
We then say that the curvature tensor of $M$ is bounded of first order.

Let $\epsilon >0$ be sufficiently small that the quasimetric
$\delta=\delta_{x,\epsilon}$ is bi-Lipschitz equivalent to a metric $d$. 
Such a constant exists by Lemma \ref{bilipmetric}.
For a fixed point
$x\in M$ we then obtain a multiplicative cross ratio on $\partial M$
by defining 
\[{\rm Cr}(\xi_1,\xi_2,\xi_3,\xi_4)=\frac{\delta(\xi_1,\xi_3)
  \delta(\xi_2,\xi_4)}{\delta(\xi_1,\xi_4)\delta(\xi_2,\xi_3)}.\]
By Lemma \ref{independence}, the cross
ratio ${\rm Cr}$ does not depend on $x$ and is invariant under
the action of the isometry group ${\rm Iso}(M)$ of $M$.
The following is an analog of Lemma 6 of \cite{Ni13}.

\begin{lemma}\label{derivative}
  For every $\phi\in {\rm Iso}(M)$ there exists a positive
  continuous function $\vert \phi^\prime\vert$ on $\partial M$
  with the property that for all $\xi,\eta\in \partial M$, we have
  \[\delta^2(\phi \xi, \phi \eta)=\vert \phi^\prime\vert(\xi)
 \vert \phi^\prime\vert(\eta)\delta^2(\xi,\eta).\]   
\end{lemma}  
\begin{proof} We copy the short proof from Lemma 6 of \cite{Ni13}
for completeness.
Let $x,u,v$ be a triple of distinct points in $\partial M$.
Since $\phi$ preserves ${\rm Cr}$, 
for any fourth distinct point $y$ we have
\begin{equation}
\frac{\delta (\phi x,\phi y)}{\delta(x,y)}
\frac{ \delta(\phi u,\phi v)}{\delta(u,v)}=
\frac{\delta(\phi x,\phi v)}{\delta(x,v)}\frac{\delta(\phi y,\phi u)}{\delta(y,u)}.
\notag\end{equation}
When $y\to x$ one obtains
\begin{equation}\label{cro}\lim_{y\to x}\frac{\delta(\phi x,\phi y)}{\delta (x,y)}=
\frac{\delta(\phi x, \phi v)}{\delta(x,v)} 
\frac{\delta(\phi x, \phi u)}{\delta (x,u)}\frac {\delta(u,v)}{\delta (\phi u,\phi v)}.\end{equation}

Let $\vert \phi^\prime \vert_{u,v}$ denote the right-hand side of 
equation (\ref{cro}), viewed as a function of $x$.
Then $\vert \phi^\prime\vert_{u,v}$ is a positive continuous function on 
$X-\{u,v\}$. Since the left-hand side of equation (\ref{cro}) does not 
depend on $u,v$, by replacing $u,v$ by a different pair of points 
we can extend $\vert \phi^\prime \vert_{u,v}$ to a positive 
function $\vert \phi^\prime \vert$ on all of $X$.
The formula in the lemma is now immediate from the definition and invariance
of ${\rm Cr}$ (see the proof of Lemma 6 of \cite{Ni13}).
\end{proof}  

Recall that the notion of a Lipschitz map makes sense
for the quasimetric 
$\delta$. 
Since there exists a constant 
$c>0$ such that $c\delta\leq d\leq \delta$, 
Lipschitz maps
for $d$ are precisely the Lipschitz 
maps for $\delta$.

The following corollary can readily be checked directly but is an immediate
consequence of Lemma \ref{derivative}. 

\begin{corollary}\label{lipschitz10}
${\rm Iso}(M)$ acts on $(\partial M, d)$ as a group of 
bi-Lipschitz transformations.
\end{corollary}

\begin{lemma}\label{lipschitz}
  For each $\phi\in {\rm Iso}(M)$
 the function $\vert \phi^\prime\vert$ is Lipschitz 
continuous for $d$.
\end{lemma}
\begin{proof}
  We follow p.779 of \cite{Ni13}. Namely, let $\xi,\eta\in \partial M$
  and choose a point $\zeta$ with $d(\xi,\zeta)\geq {\rm diam}(\partial M)/2$.
Since $d$ satisfies the 
triangle inequality, we have 
\begin{align}\label{trianglein}
\frac{d(\phi \xi, \phi \zeta)}{d(\xi,\zeta)}-\frac{d(\phi \eta, \phi \zeta)}{d(\eta,\zeta)} 
&\leq \frac{d(\phi \xi, \phi \eta)}{d(\xi,\zeta)}+\frac{d(\phi \eta,\phi \zeta)}{d(\xi,\zeta)}-
\frac{d(\phi \eta,\phi \zeta)}{d(\eta,\zeta)}\notag\\
&\leq \frac{d(\phi \xi,\phi \eta)}{d(\xi,\zeta)} +\frac{d(\phi \eta,\phi \zeta)}{d(\eta,\zeta)}
\frac{d(\xi,\eta)}{d(\xi,\zeta)} 
\end{align}

Now if $L>1$ is the Lipschitz constant for $\phi$ then 
$d(\phi \xi,\phi \eta)\leq L d(\xi,\eta)$ and additionally
$\frac{d(\phi \eta, \phi \zeta)}{d(\eta,\zeta)}\leq L$. Hence the term (\ref{trianglein}) in the above
expression is bounded from above by 
$LCd(\xi,\eta)$ 
for a universal constant $C>0$.

On the other hand, the formula for $\vert \phi^\prime \vert$ in Lemma \ref{derivative}
shows that 
\[\sqrt{\vert \phi^\prime \vert} (\xi)-
\sqrt{ \vert \phi^\prime \vert}(\eta)=\frac{1}{\sqrt {\vert \phi^\prime \vert}(\zeta)}
\bigl( \frac{d(\phi \xi,\phi \zeta)}{d(\xi,\zeta)}- \frac{d(\phi \eta, \phi \zeta)}{d(\eta,\zeta)}\bigr).\]

Together this yields 
Lipschitz continuity for $\sqrt{\vert \phi^\prime\vert}$ and hence
for $\vert \phi^\prime\vert$.
\end{proof}

By construction, we also have

\begin{lemma}\label{cocycle}
For $\phi,\psi\in \Gamma$ and all $\xi\in \partial M$ we have
\[\vert (\phi \circ \psi)^\prime\vert(\xi)= \vert \phi^\prime \vert 
(\psi(\xi)) \vert \psi^\prime\vert (\xi).\]
 \end{lemma}
\begin{proof}
The lemma is immediate from Lipschitz continuity of 
$\vert \phi^\prime \vert, \vert \psi^\prime \vert$ and 
Lemma \ref{derivative}. 
\end{proof}

Let as before $\lambda$ be the 
Lebesgue Liouville measure on the 
unit tangent bundle $T^1M$ of $M$. 
It disintegrates to an 
${\rm Iso}(M)$-invariant Radon measure $\hat \lambda$ on 
the space of geodesics $\partial M\times \partial M-\Delta$.

For $x\in M$ consider the Lebesgue measure $\lambda_x$ on 
$\partial M$ defined as the push forward  of the Lebesgue measure on 
$T_x^1M$ by the natural homeomorphism $T_x^1M\to \partial M$.
Let $\epsilon >0$ be the constant which enters the definition of the 
metric $d$. We have

\begin{proposition}\label{bounded}
There exists a uniformly bounded function $\beta$ such that 
$\hat \lambda=\beta \delta^{-2/\epsilon}\lambda_x\times \lambda_x$.
\end{proposition}
\begin{proof} By Corollary \ref{absolute}, 
the measure class of the measure 
$\lambda_x$ on $\partial M$ 
does not depend on $x$.

Let $p(x,y,\xi)$ be 
the Radon Nikodym derivative of $\lambda_y$ with respect to $\lambda_x$ 
at $\xi$ (whenever this exists). It follows from Proposition \ref{regular} and 
its proof, using 
 Lemma \ref{hoelderclose}, Corollary \ref{meancurvature} and 
the definition of the function $\delta$, that there exists a universal 
constant $C_0>0$ such that the following holds true. Let 
$\xi\not=\eta\in \partial M$ and let $y$ be the shortest distance projection of 
$x$ into the geodesic connecting $\xi$ to $\eta$; then 
\[\delta^{2/\epsilon}(\xi,\eta) p(x,y,\xi)p(x,y,\eta)\in [C_0^{-1},C_0].\]

Thus we are left with showing the following.
The measure $\hat \lambda$ is absolutely continuous with respect to 
$\lambda_x\times \lambda_x$, and there exists 
a number $C_1>0$ such that 
for any geodesic 
$\gamma\subset M$  and any $x\in \gamma$, the Radon Nikodym derivative
at $(\gamma(-\infty),\gamma(\infty))$ of $\hat \lambda$ 
with respect to $\lambda_x\times \lambda_x$, if it is exists, 
is contained in the interval $[C_1^{-1},C_1]$.

To this end
recall from the proof of Lemma \ref{symmetry} the definition of the
$(2n-2)$-form $\tilde \omega$ on $T^1M$. This form restricts to a 
volume form on a local smooth transversal $N$ for the geodesic flow,
defining a measure $\tilde \lambda$ on $N$ of the form 
$\kappa \hat \lambda$ where $\kappa$ is a continuous function with values in 
a compact subinterval of $(0,\infty)$ not depending on $N$. 

Let $x\in N$; then by perhaps decreasing the size of $N$, we may assume
that $\lambda_x\times \lambda_x$ descends to a measure on $N$
by viewing $N$ as an (open) subset of the space of geodesics. 
It now suffices to show that $\lambda_x\times \lambda_x$ is contained in the measure
class of $\tilde \lambda$, with density near $x$ bounded from above and below
by a universal positive constant. However this follows from 
Proposition \ref{regular} and its proof. 
As this is a standard argument in smooth dynamics (which does not rely on 
an underlying dynamical system), carefully laid out in Chapter III of \cite{Mn87}, 
we omit a more detailed discussion.

Together this completes the proof of the proposition.
\end{proof}

Let us summarize now what we achieved so far. For $x\in M$ consider the 
measure $\lambda_x$ on $\partial M$. Its measure class is invariant under the 
action of ${\rm Iso}(M)$. By Lemma \ref{derivative} and Lemma \ref{lipschitz}, 
for each $\phi\in \Gamma$ there exists a  
positive Lipschitz continuous
function $\vert \phi^\prime \vert$ on $\partial M$ 
which coincides with ${\rm RN}(\phi)^{1/\epsilon}$ up to a universal constant, where
${\rm RN}(\phi)$ is taken with respect to  the measure $\lambda_x$. 

More precisely, the function $\vert \phi^\prime \vert(x)$ measures the infinitesimal dilatation of 
$\phi$ with respect to the quasimetric $\delta=\delta_x$. Lemma \ref{cocycle}
shows that 
its logarithm defines a cocycle $c_\phi$ for $\Gamma$ by 
\[c_\phi(x,y)=\log \vert \phi^\prime(x)\vert - \log \vert \phi^\prime(y)\vert.\]
As the measure $\lambda_x$ is Ahlfors regular for $d_x$, 
via comparing the cocycle defined by the logarithm of Radon Nikodym derivatives 
to the cycycle $c_\phi$, 
condition $(*)$
from Section \ref{coho} is fulfilled. In particular, for sufficiently large $p$, this 
cocycle has values in $L^p(\partial M\times \partial M-\Delta,\hat \lambda)$. 

Furthermore, there is a bounded function 
$\beta$ such that $\beta\delta_x^{-1/\epsilon}\lambda_x\times \lambda_x$ is invariant
under the action of ${\rm Iso}(X)$.
Now if $\phi\in {\rm Iso}(M)$ is a loxodromic element then 
$\phi$ acts on $\partial M$ with north-south dynamics and hence 
the assumptions stated in Proposition \ref{lowerestimate}
are fulfilled for $\phi$. As a consequence, the cocycle is unbounded in 
$L^p(\partial M\times \partial M-\Delta,\hat \lambda)$ for all sufficiently large $p$
and hence it can not be a coboundary. 
Together we have shown

\begin{theorem}\label{actionon}
  Let $\Gamma$ be a discrete group which admits
a non-elementary isometric action on 
a Hadamard manifold $M$ of bounded negative curvature, with
first order bounded 
curvature tensor. Then there exists a number $p>1$ and 
an isometric action of $\Gamma$ on an 
$L^p$-space $V$ such that $H^1(\Gamma,V)\not=0$. 
\end{theorem}

For the proof of the Theorem  from the introduction, we are left
with removing the assumption on the first order bound of the covariant
derivative of $M$.

\begin{proof}[Proof of the main theorem]
Let $\Gamma\subset {\rm Iso}(M)$ acting on a Hadamard manifold $(M,g)$ of bounded 
negative curvature. By the main proposition of \cite{K05}, there exists a positive number 
$\epsilon >0$ with the following properties.
\begin{enumerate}
\item The Ricci flow $g_t$ started at $g_0=g$ exists on the time interval $[0,2\epsilon]$.
\item The curvature of $g_\epsilon$ is contained in $[-b^2,-a^2]$ for some numbers
$0<a<b$ depending on the curvature of $M$, the dimension and $\epsilon$.
\item $\Vert \nabla R_\epsilon\Vert \leq C$ where $C>0$ is a fixed constant and $R_\epsilon$ 
is the curvature tensor of $g_\epsilon$. 
\end{enumerate}

Since the Ricci flow commutes with isometries, the group $\Gamma$ admits a non-elementary
isometric action on $(M,g_\epsilon)$. Thus we can apply Theorem \ref{actionon} to 
the action of $\Gamma$ on $(M,g_\epsilon)$ and complete the proof.
\end{proof}

Since via a factor projection, a lattice in a semi-simple Lie group
$G$ of non-compact type acts 
on each rank one factor of $G$ in
a non-elementary fashion we conclude

\begin{corollary}\label{irreducible}
Let $\Gamma$ be a lattice in a semisimple Lie group 
$G$ of noncompact type containing at least one factor which is of rank one.
Then $\Gamma$ admits an isometric action on an $L^p$-space $V$ with 
$H^1(\Gamma,V)\not=0$.
\end{corollary}

\begin{remark}\label{lp}
The lower bound on $p$ we obtain from the proof of Theorem \ref{actionon}
is not sharp. For example, if $M$ is hyperbolic space of dimension $n\geq 2$, 
then the bound we find equals $p=2n-2$ while it is known that $p=2$ is possible. 
\end{remark}

\appendix
\section{Mean curvature of horospheres}

This appendix is independent of the rest
of this article, and we keep the exposition self-contained.

Throughout, we consider
a smooth Hadamard manifold $M$ of bounded negative curvature, that is,
the curvature is contained in an interval $[-b^2,-a^2]$ for numbers
$0<a\leq b$. Denote by $\partial M$ the ideal boundary of $M$.

Let $\Pi:T^1M\to M$ be the unit tangent bundle of $M$. The Levi Civita
connection on $TM$ determines a splitting 
\[TT^1M={\cal H}\oplus {\cal V}\] 
where the \emph{vertical bundle} ${\cal V}$ 
is the tangent bundle of the fibers,
and such that the restriction of $d\Pi$ to 
the \emph{horizontal bundle} ${\cal H}$ is a fiberwise isomorphism.
The decomposition is orthogonal with respect to the Sasaki metric on 
$T^1M$. For this metric, the projection $\Pi$ is a Riemannian submersion. 

For $x\in M$ and a unit vector
$v\in T^1M$ denote by $\gamma_v$ the geodesic with initial velocity 
$\gamma_v^\prime(0)=v$. 
The \emph{geodesic spray} $X$ is the generator of the geodesic flow 
$\Phi^t:T^1M\to T^1M, v \to \Phi^tv=\gamma_v^\prime(t)$. It is a section of 
the bundle ${\cal H}$. Thus ${\cal H}$ decomposes as an orthogonal 
direct sum ${\cal H}=E\oplus \mathbb{R}X$ where the fiber of 
$E$ at $v$ is mapped by $d\Pi$ onto $v^\perp$. 

Let $A(v)$ be the shape operator 
at $x$ of 
the \emph{horosphere} through $x$ defined by the boundary point 
$\gamma_v(-\infty)\in \partial M$. Since horospheres are of class $C	^2$, 
this is a well defined symmetric linear endomorphism of $v^\perp$. 
Using the identification of the fiber $E_v$ of $E$ at $v$ with $v^\perp$, 
we view $A(v)$ as a symmetric 
linear endomorphism of $E_v$. Thus 
$v\to A(v)$ is a symmetric section of the bundle $E^*\otimes E$. 
The bundle $E^*\otimes E$ in turn is naturally equipped with a smooth
Riemannian metric constructed from the Levi Civita connection of the 
Sasaki metric.

The following is the main result of this appendix. It was established in \cite{Ho40}
for surfaces, that is, in the case that the dimension $n$ of $M$ equals $2$. It is also
well known under the assumption that $M$ admits a cocompact isometry group
(where boundedness of $\nabla R$ is automatic), using tools from 
smooth dynamics. We refer to \cite{Mn87} for 
details. 

\begin{proposition}\label{hoelder2}
If the covariant derivative $\nabla R$ of the curvature tensor is 
uniformly bounded in norm then the section $v\to A(v)$ of 
$E^*\otimes E$ is H\"older continuous. 
\end{proposition}

For $v\in T^1M$, the mean curvature $m(v)$ of the horosphere 
defined by $v$ equals the trace of the shape operator $A(v)$.
Thus as an immediate corollary of Proposition \ref{hoelder2} we obtain

\begin{corollary}\label{meancurvature} 
If the covariant derivative $\nabla R$ of the curvature tensor is 
uniformly bounded in norm then 
the function $m:v\to m(v)$ on $T^1M$ is H\"older continuous.
\end{corollary}

The strategy of proof consists in comparing the shape operator of 
horospheres with the shape operator of hypersurfaces depending
smoothly on the defining data. More precisely, 
for $R>0$ let $A_R(v)$ be the shape operator
at $x=\gamma_v(0)$ of the hypersurface
\[N_R(v)=\{y\mid d(y, \exp  (\gamma_v^\prime(-R)^\perp))=R\}\]
where $\exp$ denotes the exponential map of $M$, 
and define similarly $A^S_R(v)$ to be the shape operator
at $x=\gamma_v(0)$ of the distance sphere 
of radius $R$ about $\gamma_v(-R)$. 

In the statement of the following lemma, norms are taken with 
respect to the natural Riemannian metric on $E^*\otimes E$.
Only bounded negative curvature on $M$ is necessary for Lemma \ref{hoelderclose}
and Lemma \ref{controlq} 
to hold true, that is, no assumption on the covariant derivative
of the curvature tensor is required.

\begin{lemma}\label{hoelderclose}
There exist numbers $C_0>0,\alpha >0$ only depending on the 
curvature bounds such that 
\begin{equation}
\vert A(v)-A_R(v)\vert  \leq C_0e^{-\alpha R} \text{ and }
\vert A(v)-A^S_R(v) \vert  \leq C_0e^{-\alpha R} \notag\end{equation}
for all  $v\in T^1M$ and 
$R\geq 10$.
\end{lemma}
\begin{proof} Let $x\in M,v\in T_x^1M$ and 
let $H$ be the horosphere passing though $x$ which is 
determined by the ideal boundary point $\gamma_v(-\infty)\in \partial M$. 
The shape operator of $H$ at $x$ can be computed 
as follows. Let $X\in T_xH=v^\perp$ be a tangent vector of $H$ at $x$.
Then $X$ determines uniquely a Jacobi field $J_X$ along $\gamma_v$
with $J_X(0)=X$ and $\lim_{t\to -\infty}\Vert J_X(t)\Vert =0$.  
The shape operator $A(v)$ then equals the linear endomorphism
$X\to \nabla J_X(0)$ 
of the euclidean vector space $T_xH=v^\perp$.
Here $\nabla J_X$ denotes the covariant derivative of $J_X$ along the 
geodesic $\gamma_v$.

Similarly, for $R\geq 0$ the shape operator
$A_R(v)$ is computed as the linear map 
$X\to \nabla J_X^R(0)$ where $J_X^R$ is the Jacobi field
along $\gamma_v$ with 
$J_X^R(0)=X\in v^\perp$ and $\nabla J_X^R(-R)=0$. It now suffices to show 
that 
\begin{equation}\label{shape1} 
\Vert \nabla J_X(0)- \nabla J_X^R(0)\Vert \leq 
Ce^{-\alpha R} \Vert X\Vert\end{equation}
for constants $\alpha >0, C>0$ only depending on the curvature bounds.

For each $X$ consider the Jacobi field 
$\hat J_X=J_X-J_X^R$ along $\gamma_v$. 
It vanishes at $t=0$.
The Rauch comparison theorem shows 
that 
\begin{equation}\label{rauch1}
\Vert \hat J_X (-R)\Vert \geq  \sinh aR\, \Vert 
\nabla \hat J_X(0)\Vert\end{equation}
(here as before, $-a^2$ is an upper curvature bound for $M$).
Furthermore, there exists a constant $C_0>0$ only depending on the 
curvature bounds such that
\begin{equation}\label{rauch3}
\Vert \nabla \hat J_X(-R)\Vert \geq C_0 \Vert \hat J_X(-R)\Vert \text{ for all }R.
\end{equation}

On the other hand, we have
$\nabla \hat J_X(-R)=\nabla J_X(-R)$. Hence using once more comparison,
we obtain 
\begin{equation}\label{rauch2}
\Vert \nabla \hat J_X(-R)\Vert \leq C_1e^{-a R}\Vert X\Vert \end{equation}
for a universal constant $C_1>0$.
For $R>10$ the estimates (\ref{rauch1}), (\ref{rauch3}) 
and (\ref{rauch2}) together yield that indeed, 
\[\Vert \nabla J_X(0)-\nabla J_X^R(0)\Vert \leq C_2e^{-a R}
\Vert \hat J_X(-R)\Vert \leq C_2C_0^{-1}C_1 e^{-2aR}\Vert X\Vert\]
for a universal constant $C_2>0$. 

This shows the first estimate stated in the lemma. The second estimate
follows from exactly the same argument, replacing the condition 
$\nabla \hat J_X(-R)=\nabla J_X(-R)$ by the condition 
$\hat J_X(-R)=J_X(-R)$. The lemma follows. 
\end{proof}

The principal bundle ${\cal P}\to M$ of orthonormal frames
in $TM$ is equipped with 
the Levi Civita connection which defines a decomposition 
$T{\cal P}={\cal H}\oplus {\cal V}$ where ${\cal V}$
is the tangent bundle of the 
fibers (note that this splitting is related to the splitting of 
$TT^1M$, but the fiber spaces are different. We nevertheless use the 
same notation here to keep the notations simple). 
This splitting determines a smooth
Riemannian metric on ${\cal P}$ with the 
following properties. 
The fibers are isometric to the orthogonal group with the bi-invariant metric
defined by the Killing form,   
the decomposition $T{\cal P}={\cal H}\oplus {\cal V}$ is orthogonal, and 
${\cal P}\to M$ is a Riemannian submersion. We denote by $d_{\cal P}$ the 
distance for this metric. 

Let $x\in M$ and let $v\not=w\in T_x^1M$ be
two unit tangent vectors based at $x$.
Denote by $\angle (v,w)$ the euclidean angle between $v,w$ and assume that 
$\angle(v,w)<\pi/4$. 
Choose an orthonormal basis 
$P(0,0)=(e_1,\dots,e_{n-1},v)$ of $TM$ at $x$ with the property
that $e_{n-1}$ is contained in the plane spanned by $v,w$. 
Let $s\to \chi(s)\in T_x^1M$ $
(s\in [0,\angle(v,w)])$ be the shortest geodesic connecting $v$ to $w$ 
(which is contained in the plane spanned by $v,w$) and 
define 
\[\zeta(s,t)=\gamma_{\chi(s)}(t).\]
Thus $\zeta$ is a variation of geodesics through the point $x$.

Let $Q(s,t)$ be the frame over $\zeta(s,t)$ obtained from 
$P(0,0)$ by parallel transport along $\zeta_s(t)=\zeta(s,t)$.
The following lemma is probably well known and 
included here for completeness. In its formulation,
norms of tangent vectors are taken with respect to the Riemannian
metrics on ${\cal P}$ and $M$.

\begin{lemma}\label{controlq}
  There exists a number $C_1>0$
  such that $\Vert \frac{\partial}{\partial s} Q(s,t)\Vert 
 \leq  C_1 \Vert \frac{\partial}{\partial s}\zeta(s,t)\Vert$
 for all $s,t$.
\end{lemma}
\begin{proof}
The map $(s,t)\to Q(s,t)$ is a variation of horizontal geodesics in ${\cal P}$
with the same starting point $Q(0,0)$. 
For each $t$, the path $Q_t:s\to Q(s,t)$ is 
a lift to ${\cal P}$ 
of the path $\zeta_t:s\to \zeta(s,t)$ in $M$. It is 
smooth but may not be horizontal. Its tangent can be 
decomposed as 
\[Q_t^\prime(s)= Q_{t,{\cal H}}^{\prime}(s)+
Q_{t,{\cal V}}^\prime(s)\] 
into a horizontal and vertical component. Since 
$\Vert Q_{t,{\cal H}}^\prime(s)\Vert= \Vert \zeta_t^\prime(s)\Vert$ 
we have to show the existence of a number $C>0$ such that
\begin{equation}\label{verticalder}
\Vert Q^\prime_{t,{\cal V}}(s)\Vert \leq C\Vert \zeta_t^\prime(s)\Vert\end{equation}
for all $s,t$.  That this holds indeed true 
can be seen as follows.

Let $\omega$ be the connection $1$-form on ${\cal P}$. This is a one-form
on ${\cal P}$ with values in the Lie algebra
${\mathfrak so}(n)$ of the structure group $O(n)$ of ${\cal P}$ which vanishes
on the horizontal bundle ${\cal H}$.
The ${\mathfrak so}(n)$-valued curvature form 
$\Omega=d\omega +\frac{1}{2}[\omega,\omega]$
is horizontal, that is, it is annihilated by ${\cal V}$.
If $W\to M$ denotes the vector bundle 
of antisymmetric linear endomorphisms of $TM$ 
then $\Omega$ 
descends to the $W$-valued $2$-form on $M$ defined by the
Riemannian curvature tensor.

Let $s\geq 0$ be fixed. For small $h>0$ 
consider the piecewise smooth 
loop $\zeta_{s,h,t}$ in $M$ based at $\zeta(s,t)$ 
which is the concatentation of the 
subsegment $u\to \zeta(s+u,t)$ 
of $\zeta_t$ connecting $\zeta_t(s)$ to 
$\zeta_t(s+h)$, the geodesic
arc $u\to \zeta(s+h,t-u+h)$ and the geodesic arc 
$u\to \zeta(s,u-t-h)$.  
Let $Q_{s,h,t}$ be the lift of $\zeta_{s,h,t}$ to the (bordered) surface 
\[Q=\{Q(s,t)\mid 0\leq s\leq \angle(v,w),0\leq t\}.\] This is a
piecewise smooth closed curve in ${\cal P}$ which bounds a subsurface of $Q$.

Since the connection form $\omega$ vanishes on ${\cal H}$ and the
arcs $t\to Q(s,t)$ are horizontal geodesics, 
the integral of 
$\omega$ over this piecewise smooth loop equals the element 
$a(s,h)=\int_s^{s+h} \omega(Q^\prime_{t,{\cal V}}(u))du\in
{\mathfrak so}(n)$. Since $Q$ is a smooth map,
for sufficiently small $h$ we have 
\begin{equation}\label{q}
h\Vert Q^\prime_{t,{\cal V}}(s)\Vert \leq 2\Vert a(s,h)\Vert\end{equation}
(here the norm is taken in the Lie algebra
${\mathfrak so}(n)$).

By comparison, for small enough $h$ 
the area of the sector
$(u,z)\to \zeta(u,z)$ $(s\leq u\leq s+h, 0\leq z\leq t)$ with respect
to the pull-back of the metric on $M$ is
bounded from above by $C_0h\Vert \zeta_t^\prime (s)\Vert $ 
where $C_0>0$ is a constant only depending on the curvature bounds.
Since the 
curvature tensor, viewed as a symmetric bundle map $W\to W$, 
is pointwise uniformly bounded in norm, the integral of the norm
of the curvature tensor 
over this surface is bounded from above by 
$C_1 h\Vert \zeta_t^\prime(s)\Vert$ for a universal constant $C_1>0$.

On the other hand, 
$\omega$ vanishes on ${\cal H}$ and the curves
$t\to Q(s,t)\in {\cal P}$ are horizontal geodesics. Thus we have
$\Omega\vert Q=d\omega \vert Q$. Using 
Stokes's theorem, this implies that
$\Vert a(s,h)\Vert \leq C_2 h\Vert \zeta_t^\prime(s)\Vert$ 
where $C_2>0$ is a universal constant.
Using the inequality (\ref{q}) and taking the limit as 
$h\searrow 0$ 
yields the estimate (\ref{verticalder}). From this the 
lemma follows. 
\end{proof}

Let $P(s,0)=(e_1,\dots,e_{n-2}, e_{n-1}(s),\chi(s))$ be the frame 
obtained by rotating the plane
spanned by $v,w$ keeping the orthogonal complement pointwise fixed.
The frame $P(s,0)$ extends by parallel transport along the geodesics 
$\zeta_s:t\to \zeta(s,t)$ to a section 
$(s,t)\to P(s,t)$ of ${\cal P}$ over $\zeta$.
The map 
$(s,t)\to P(s,t)\in {\cal P}$ is a smooth embedding. Each of the curves
$t\to P(s,t)$ is a horizontal geodesic. 
Furthermore, by the definition of the 
Sasaki metric, we have 
\[d_{\cal P}(Q(s,t),P(s,t))=s\text{  for all }s,t\] 
and
hence Lemma \ref{controlq} shows that 
\begin{equation}\label{pestimate}
\Vert \frac{\partial}{\partial s}P(s,t)\Vert \leq 1+C_1\Vert 
\frac{\partial}{\partial s}\zeta(s,t)\Vert. 
\end{equation}

We are now ready to complete the  proof of Proposition \ref{hoelder2}. 
We now need to assume that the curvature of $M$ is bounded between two 
negative constants and that the 
covariant derivative $\nabla R$ of the 
curvature tensor is uniformly bounded in norm.

\begin{proof}[Proof of Proposition \ref{hoelder2}] By the assumption on $M$, 
for each $\xi\in \partial M$ the gradient field ${\rm grad}\,b_\xi$ of a 
Busemann function $b_\xi$ 
at $\xi$ is a section of $T^1M$ of class $C^2$, with uniformly bounded 
first and second covariant derivatives \cite{Shc83}. 
For $x\in M$ and 
$v={\rm grad}\, b_\xi (x)$, the shape operator
$A(v)$ at $x$ 
of the horosphere $b_\xi^{-1}(b_\xi(x))$ 
equals the linear map 
\[X\in v^\perp  \to \nabla_X {\rm grad}\, b_\xi\] 
where in all computations, we normalize shape operators to 
be positive semi-definite. 
Thus the restriction of the section $A$ of $E^*\otimes E$ to the image    
of the section ${\rm grad}\,b_\xi$ 
is of class $C^1$, 
with pointwise uniformly bounded differential with respect to 
the Sasaki metric, and hence it is H\"older continuous. 

For $v\in {\rm grad} \,b_\xi$, the tangent space of $T^1M$ at $v$ is a direct
sum of the vertical tangent space ${\cal V}_v$, that is, the tangent
space of the fibers of the fibration $T^1M\to M$, and the tangent space
$T_v{\rm grad}\, b_\xi$ 
of the $C^2$-submanifold ${\rm grad}\, b_\xi$. Due to the fact that the 
eigenvalues of the shape operators of horospheres are bounded from above
and below by universal positive constants, 
this decomposition is well adapted to the Sasaki metric. By this we mean
that the angle between a vector of ${\cal V}_v$ and a vector of 
$T_v {\rm grad}\,b_\xi$ is bounded from below by a universal positive constant. 
Moreover, any two points $v,w\in T^1M$ can be connected by a piecewise smooth 
path which consists of finitely many segments alternating between 
segments in submanifolds ${\rm grad}\,b_\xi$ for some $\xi\in \partial M$ and
segments in fibers of the fibration $T^1M\to M$ and whose length is bounded from 
above by a universal constant times the distance in $T^1M$ between $v,w$. 
It therefore suffices 
to show the existence of numbers $C>0,\kappa \in (0,1)$ with the
following property. Let $x\in M$ and let $v,w\in T_x^1M$; then 
\[ \vert A(v)-A(w)\vert \leq C \angle (v,w)^\kappa\]
where as before, $\angle (v,w)$ is the Euclidean angle between the unit vectors $v,w$
and the norm is taken as the norm of a symmetric linear
endomorphism of the Euclidean vector space $T_xM$.
For ease of notations, we shall show $\vert A(-v)-A(-w)\vert 
\leq C\angle(v,w)^\kappa$.

To see that this indeed holds true let $v\not=w\in T_x^1M$ 
and assume without loss of generality that $\angle (v,w)<1/2$
and hence $\log \angle (v,w) <0$. We use now the constructions
and notations from the beginning of this appendix and consider 
the minimal geodesic 
$\chi:[0,\angle (v,w)]\to T_x^1M$ 
which connects $v$ to $w$ and the corresponding variation of 
geodesics $\zeta(s,t)=\gamma_{\chi(s)}(t)$
with variation Jacobi fields $J_s$, determined by the 
initial condition $J_s(0)=0$
and $\nabla J_s(0)=\frac{d}{ds}\chi(s)$. 
Since $\Vert \nabla J_s(0)\Vert =1$ for all $s$,
standard Jacobi 
field estimates show that 
\begin{equation}\label{comparejac}\Vert J_s(-R)\Vert \in [\sinh{aR},\sinh{bR}].
\end{equation}

Let $R>0$ be such that $\int_0^{\angle (v,w)}\Vert \zeta_R^\prime(s)\Vert ds=1$.
By the estimate (\ref{comparejac}) we have
\[\angle (v,w) \sinh aR  \leq 1 \leq \angle(v,w) \sinh bR\]
and consequently 
\begin{equation}\label{star}
R\in [-\log \angle(v,w)/b, -\log \angle (v,w)/a+C_2].\end{equation}
for a universal constant $C_2>0$.


Let $r=aR/10 b$. Using the estimate (\ref{comparejac}), 
we have 
\[d(\gamma_v(r),\gamma_w(r))\leq e^{br} \angle (v,w).\]
But $e^{br}\leq e^{a R/10} \leq e^{C_2/10}\angle(v,w)^{-1/10}$
by the estimate (\ref{star}) and therefore 
\begin{equation}\label{restimate}
e^{br}\angle (v,w)\leq e^{C_2/10}\angle (v,w)^{1-1/10}.\end{equation}
By hyperbolicity and the estimate (\ref{pestimate}), 
we also have
\begin{equation}\label{distance5}
d_{\cal P}(P(0,t),P(\angle(v,w),t)) 
\leq
C_3\angle (v,w)^{9/10}\end{equation}
 for all $t\in [0,r]$ and a universal
constant $C_3>0$, where $P(s,t)\in {\cal P}$ is as in the 
construction preceding this proof.

Using the trivialization of $TM\vert \zeta(s,t)$ defined by the 
frames $P(s,t)$, the Jacobi equation 
translates into the Riccati equation 
\begin{equation}\label{riccati}
A_s^\prime(t)+A_s^2(t)+R_s(t)=0\end{equation}
for the shape operators $A_s(t)$ of the 
hypersurfaces of distance $t$ to $\exp (\zeta_s^\prime(r)^\perp)$.
Here $R_s(t)Y=R(Y,\zeta_s^\prime(r-t))\zeta_s^\prime(r-t)$
and the equation is thought of as an ODE for symmetric $(n-1,n-1)$-matrices written with 
respect to the parallel orthonormal trivialization $t\to P(s,t)$ of
$TM\vert \zeta_s$. The solution we are looking for is 
determined by the initial condition $A_s(0)=0$. Note that for these
expressions, we invert the time of the geodesics $t\to \zeta(s,t)$ and use
a time shift so that $t=0$ corresponds to $\zeta(s,r)$ for all $s$.

The symmetric linear operators $R_s(t)$ 
of the euclidean vector spaces $\zeta_s^\prime(r-t)^\perp$ 
are uniformly bounded and uniformly 
negative definite. Since the 
covariant derivative $\nabla R$ is pointwise
uniformly bounded, it follows from the 
estimate (\ref{distance5}) 
that there exists a number $C_4>0$ not
depending on $(v,w)$ such that 
\[R_0(t)(1+C_4\angle(v,w)^{9/10}) \leq 
R_{\angle(v,w)}(t)\]
for all $0\leq t\leq r$. This means that $R_{\angle(v,w)}(t)-R_0(t)(1-C_4\angle (v,w)^{9/10})$
is non-negative definite. 

Namely, by the inequality (\ref{distance5}), in the frame bundle ${\cal P}$,
the length of the lift to ${\cal P}$ of the path $s\to \zeta_s(t)$ is bounded from 
above by $C_3\angle(v,w)^{9/10}$. The curvature form on ${\cal P}$ is a 
smooth 2-form on ${\cal P}$ with values in the Lie algebra $\mathfrak{so}(n)$ of the 
fiber group $SO(n)$. The natural Riemannian metric on ${\cal P}$ 
induces a metric on the bundle of $\mathfrak{so}(n)$-valued $2$-forms on ${\cal P}$.
Since the covariant derivative of the curvature tensor of $M$ is pointwise uniformly
bounded, the same holds true for the covariant derivative of this
$\mathfrak{so}(n)$-valued $2$-form. As a consequence, given the value of the form
at the frame $P(0,t)$, the value at the frame $P(s,t)$ differs in norm from the value 
at the frame obtained by parallel transport along the lift of the path 
$s\to \zeta(s,t)$ by a uniform multiple of the length of the path, that is, by at most
$C\angle(v,w)^{9/10}$ for a fixed number $C>0$.
Viewing the curvature as a symmetric endomorphism of the 
bundle of two-forms on $M$, if the restriction of such an endomorphism $\Lambda$
to a codimension one subspace $V$ is negative definite, with largest eigenvalue 
bounded from above by a fixed constant $-a<0$, then for any symmetric endomorphism
$T$ such that the operator norm of $\Lambda-T$ is sufficiently small,  
the restriction of $\Lambda(1+C\Vert \Lambda -T\Vert)- T$ 
to $V$ is nonpositive definite provided that $C>0$ is sufficiently large. 

As a consequence, the
symmetric matrix 
\[
(1+C_4\angle(v,w)^{9/10})R_{0}(t)-R_{\angle(v,w)}(t)\] 
is nonpositive definite for all
$0\leq t\leq r$, where we use the frames $P(s,t)$ to identify
the vector spaces $\zeta_s^\prime(t)^\perp$ with a single
euclidean vector space of dimension $n-1$.


Denote by ${\cal J}(t)$ the matrix in the parallel 
frame $P(0,t)$ defining the Jacobi fields $J_X$ along 
the geodesic $t \to \gamma_v(r-t)$ with initial condition 
$J_X(0)=X,\nabla J_X(0)=0$ where $X\in \gamma_v^\prime(r)^\perp$.
The matrix valued curve $t\to {\cal J}(t)$ consists of invertible matrices
starting at the identity 
and hence it can be written as $t \to \exp (V(t))$ where $\exp$ is the 
exponential map of the Lie group $GL(n-1,\mathbb{R})$ (for right invariant
vector fields) and where
$V(t)\in \mathfrak{gl}(n-1,\mathbb{R})$.

Put $q=1+C_4\angle(v,w)^{9/10}$ and $Q(t)=\exp (qV(t))$. 
Then $Q^\prime(t)=qV^\prime(t) Q(t)$
and consequently 
\begin{equation}\label{shape}
B(t)=Q^\prime(t)Q(t)^{-1}= q {\cal J}^\prime(t){\cal J}(t)^{-1}.\end{equation}
Since ${\cal J}^\prime(t) {\cal J}(t)^{-1}=A_0(t)$, 
we have
\[B^\prime(t)= q A_0^\prime(t)=Q^{\prime\prime}(t)Q^{-1}(t)-
(Q^\prime(t)Q(t)^{-1})^2=Q^{\prime\prime}(t)Q^{-1}(t)-
q^2 A_0(t)^2\]
and hence
\begin{equation}\label{positivelower}
B^\prime(t)+B(t)^2= qA_0^\prime(t)+q^2A_0^2(t)\geq 
q(A_0^\prime(t)+A_0^2(t))=-qR_0(t).\end{equation}
Here the inequality in the expression (\ref{positivelower}) 
stems from the fact that the matrix $A_0(t)$ is  
positive semi-definite for all $t$.

Since $qR_0(t)\leq  R_{\angle (v,w)}(t)$ for all $t$, we conclude from comparison
of solutions of 
the Riccati equation (the main theorem of 
\cite{EH90}) with the same initial condition that 
\[B(t)\geq A_{\angle(v,w)}(t)\text{ for } 0\leq t\leq r.\] 

The equation (\ref{shape}) shows that
$B(t)=qA_0(t)$ for all $t$ and hence
\[A_{\angle(v,w)}(r)\leq q A_0(r).\]
As $A_{\angle(v,w)}(r)=A_r(w)$ and $A_0(r)=A_r(v)$ 
 (via the change of coordinates defined 
by the frames $P(0,0)$ and $P(\angle(v,w),0)$)
and as for $r>10$ the eigenvalues of the shape operators $A_r(u)$ are bounded from 
above and below by a universal positive constant, 
exchanging the roles of $v,w$ yields that 
\[\vert A_r(v)-A_r(w)\vert \leq C_5\angle(v,w)^{9/10}\]
for a universal constant $C_5>0$.

 On the other hand, Lemma \ref{hoelderclose} shows that 
 $\vert A(v)-A_r(v)\vert \leq e^{-\alpha r}$ provided that $r>10$. 
 Since $r\geq - a \log \angle(v,w)/10b^2$,  together we obtain
 \begin{align} 
 \vert A(v)-A(w)\vert & \leq \vert A(v)-A_r(v)\vert +
 \vert A_r(v)-A_r(w)\vert +\vert A_r(w)-A(w)\vert \notag\\
 & \leq C_6\angle (v,w)^{\chi} \notag \end{align}
for universal constants $C_6>0,\chi>0$ 
which is what we wanted to show.
 \end{proof}

\begin{remark}\label{shape2}
By Proposition \ref{hoelder2}, 
for any $x\in M$ the shape operator of the horosphere defined
by $v\in T_x^1M$, viewed as symmetric linear operator 
on $v^\perp$, depends in a H\"older continuous fashion 
on $v$. This is equivalent to stating that  the subbundle $TW^{su}$ of 
$TT^1M$ whose fiber at $v\in T^1M$ equals the 
tangent space of the submanifold ${\rm grad}\,b_\xi$ 
$(\xi=\gamma_v(-\infty))$ is H\"older continuous. We refer to 
the explicit description of $TW^{su}$ 
in the proof of Lemma \ref{symmetry} which immediately yields this
equivalence. 
In the case that $M$
is a surface, this is the result established in \cite{Ho40}.

For H\"older continuity of $TW^{su}$  to hold true, 
the assumption that $\nabla R$ is uniformly bounded can not be omitted. 
Namely, it was shown in \cite{BBB87} that 
for every $\alpha >0$ and every $\epsilon >0$ 
there exists a surface $S_0$ with a smooth metric of finite
volume and curvature in $[-1-\epsilon,-1+\epsilon]$ 
and with the property that the subbundle $TW^{su}$ of 
$TT^1S_0$ 
is not H\"older continuous with exponent $\alpha$  
\cite{BBB87}. \end{remark}

\bigskip\bigskip\noindent
Mathematisches Institut der Universit\"at Bonn\hfil\break
Endenicher Allee 60\hfil\break
53115 Bonn, Germany

\smallskip\noindent
e-mail: ursula@math.uni-bonn.de\hfil\break
\smallskip\noindent

\smallskip\noindent

\end{document}